\documentclass[11pt]{article}
\usepackage{dsfont}
\usepackage{graphicx}
\usepackage{hyperref}
\usepackage[numbers,sort&compress]{natbib}
\usepackage{hypernat}
\usepackage{mathrsfs}
\usepackage{amsmath,amsthm,amssymb,bm}

\usepackage{indentfirst}
\usepackage[a4paper,top=2.5cm,left=2.5cm,right=2.5cm,bottom=2.5cm]{geometry}
\newtheorem{thm}{\hspace{4mm} Theorem}

\newtheorem{rem}{\hspace{4mm} Remark}

\begin{document}

\title{A decomposition method to construct cubature formulae of degree 3\thanks{The project is
supported by NNSF of China(Nos. 61033012,11171052,11271060,61272371)
and also supported by ``the Fundamental Research Funds for the
Central Universities''.}}
\author{Zhaoliang Meng\thanks{Corresponding author. E-mail: m\_zh\_l@yahoo.com.cn}
\\[4pt]
School of Mathematical Sciences, Dalian University of Technology,
Dalian, 116024, China\\[5pt]
Zhongxuan Luo\\[4pt]
School of Mathematical Sciences, Dalian University of Technology, Dalian, 116024, China\\
School of Software, Dalian University of Technology, Dalian, 116620, China}
\date{}
 \maketitle
\begin{abstract}
Numerical integration formulas in  $n$-dimensional Euclidean space
of degree three are discussed. For the integrals with permutation
symmetry we present a method to construct its third-degree
integration formulas with $2n$ real points. We present a
decomposition method and only need to deal with $n$ one-dimensional
moment problems independently.
\end{abstract}
\begin{quote}
\end{quote}
\section{Introduction}
Let $\Pi^n=\mathbb{R}[x_1,\ldots,x_n]$ be the space of polynomials
in $n$ real variables and $\mathscr{L}$ be a square positive linear
functional defined on $\Pi^n$ such as those given by
$\mathscr{L}(f)=\int_{\mathbb{R}^n}f(x)W(x)\mathrm{d}x$, where $W$
is a nonnegative weight function with finite moments of all order. Let $\Pi^n_d$ be the space of polynomials of degree at most $d$.
Here we discuss numerical integration formulas of the form
\begin{eqnarray}\label{Eq:int}
\mathscr{L}(f)\thickapprox \sum_{k}a_kf(u^{(k)}),
\end{eqnarray}
where $a_k$ are constants and $u^{(k)}$ are points in the spaces.
The formulas are called degree of $d$ if they are exact for
integrations of any polynomials of $x$ of degree at most $d$ but not
$d+1$.

In this paper, we only deal with the construction of third-degree cubature
formula. It
looks like a simple problem, however it remains to be solved. The
well known result is due to Stroud
\cite{stroud1957,stroud1960,stroud1971}. He presented a method to
construct numerical integration formulas of degree 3 for centrally
symmetrical region and recently Xiu \cite{xiu} also considered the
similar numerical formulas for integrals as
\begin{eqnarray*}
\mathscr{L}(f)=\int_{a}^b\ldots\int_a^bw(x_1)\ldots,w(x_n)f(x_1,x_2\ldots,x_n)\mathrm{d}x_1\ldots\mathrm{d}x_n.
\end{eqnarray*}
In \cite{xiu}, Xiu assumed that every single integral is symmetrical, which means his result naturally belongs to centrally
symmetrical case. Recently,  the authors \cite{meng} extended Stroud's results and
presented formulas of degree 3 of $2n$ points or $2n+1$ points for
integrals as
\begin{eqnarray*}
 && \mathscr{L}(f)=\int_{a_1}^{b_1}\ldots\int_{a_n}^{b_n}w_1(x_1)\ldots
  w_n(x_n)f(x_1,\ldots,x_n)\ \mathrm{d}x_1\ldots \mathrm{d}x_n,\\
&&  w_i(x_i)\geq 0,\quad  x_i\in [a_i,b_i],\quad i=1,\ldots,n.\notag
\end{eqnarray*}
Besides, many scholars employed the invariant theory method to deal
with symmetrical case and we can refer to
\cite{cools2001,cools2002,cools2003} and the references therein. As
far as we know, $2n$ is the minimum of the integration points except
some special regions(see \cite{stroud1961}), and for centrally
symmetrical region Mysovskikh \cite{mysovskikh} had shown this
point. However, for the general integration case, it remains unknown
how to construct the formulas of degree 3 with $2n$ points. In the
two-dimensional case, third degree integration formulas with 4 real
points was given in \cite{gunther,luo} for any regions. But it is
difficult to extend it to higher dimension. For other related work,
we can refer to \cite{Om,meng2,erich} and the reference therein.

This paper will extend the results in \cite{meng} for the integrals of product regions to
those with permutation symmetry.  First we present a
condition which is satisfied by the integral. And then we prove that under
this condition, the construction problem of cubature formulae with degree
three can be transformed into two smaller sub-cubature problems. Finally, for
the construction of cubature rules of the integrals with permutation symmetry
can be decomposed into $n$ one-dimensional moment problems.

This paper is organized as follows. The construction of cubature
formulas of degree 3 are presented in section 2. And section 3 will
present two examples to illustrate the construction process.
Finally, section 4 will make a conclusion.
\section{The construction of third-degree formulas}
Assume that $\mathscr{L}$ has the following property: \vspace{6pt}

\begin{minipage}[c]{0.1\textwidth}
    (P)
\end{minipage}
\begin{minipage}[c]{0.8\textwidth}
 \emph{There exist $n$ linearly independent polynomial
  $l_i(x_1,\ldots,x_n),i=1,2,\ldots,n$ such that all $l_il_n(i=1,\ldots,n-1)$
  are the orthogonal polynomials of degree two with respect to
  $\mathscr{L}$.}
\end{minipage}

Let
\begin{eqnarray}\label{eq:transformation}
T:\quad l_i(x_1,\ldots,x_n)\longrightarrow t_i, \quad i=1,2,\ldots,n
\end{eqnarray}
be a linear transformation and $\mathscr{L}$ be transformed into $\mathscr{L}'$. Then by
the assumption all $t_it_n(i=1,2\ldots,n-1)$ are the orthogonal polynomial of degree two
with respect to $\mathscr{L}'$. Here we do not require that all $t_it_n(i=1,2\ldots,n-1)$
can constitute a basis of orthogonal polynomials of degree 2 with respect to
$\mathscr{L}'$. We can also assume that the third-degree formula of $\mathscr{L}'$ has
the following form
\begin{eqnarray}\label{eq:cub_form}
\begin{split}
 &v^{(1)}=(v_{1,1},v_{1,2},v_{1,3}\ldots,v_{1,n-1},0)\quad \omega_{1}\\
 &v^{(2)}=(v_{2,1},v_{2,2},v_{2,3}\ldots,v_{2,n-1},0)\quad \omega_{2}\\
  &\cdots \cdots \cdots\cdots \cdots \cdots \cdots \cdots \cdots\cdots \\
  &v^{(N)}=(v_{N,1},v_{N,2},v_{N,3}\ldots,v_{N,n-1},0)\quad \omega_{N}\\
  &v^{(N+j)}=(0,0,0\ldots,0,v_{N+j,n})\quad \omega_{N+j},\ j=1,2.
 \end{split}
\end{eqnarray}
To enforce polynomial exactness of degree 3, it suffices to require
\eqref{eq:cub_form} to be exact for
  $$
1, \ t_1,\ t_2,\ldots,t_n,\ t_it_j,\ t_it_jt_k\quad
i,j,k=1,2,\ldots,n.
  $$
Then we have
\begin{subequations}\label{eq:moment_equation}
\begin{eqnarray}
&&\quad \omega_1+\omega_2+\ldots+\omega_{N+2}=\mathscr{L}'(1) \label{eq:moment_equation_constant}\\
&&\left\{
  \begin{array}{ll}
\omega_1v_{1,i}+\omega_2v_{2,i}+\ldots+\omega_{N}v_{N,i}=\mathscr{L}'(t_i),\quad i=1,2,\ldots,n-1 \\
\omega_{N+1}v_{N+1,n}+\omega_{N+2}v_{N+2,n}=\mathscr{L}'(t_n)
\end{array}
\right. \\
&&\left\{
  \begin{array}{ll}
\omega_1v_{1,i}v_{1,j}+\omega_2v_{2,i}v_{2,j}+\ldots+\omega_{N}v_{N,i}v_{N,j}=\mathscr{L}'(t_it_j),\quad i,j=1,2,\ldots,n-1 \\
 \omega_{N+1}v_{N+1,n}^2+\omega_{N+2}v_{N+2,n}^2=\mathscr{L}'(t_n^2)
\end{array}
\right. \\
 &&\left\{
  \begin{array}{ll}
\omega_1v_{1,i}v_{1,j}v_{1,k}+\ldots+\omega_{N}v_{N,i}v_{N,j}v_{N,k}=\mathscr{L}'(t_it_jt_k),\quad i,j,k=1,2,\ldots,n-1 \\
\omega_{N+1}v_{N+1,n}^3+\omega_{N+2}v_{N+2,n}^3=\mathscr{L}'(t_n^3)
\end{array}
\right.
\end{eqnarray}
\end{subequations}
and the equation \eqref{eq:moment_equation_constant} can be rewritten as
\begin{eqnarray*}
&&\omega_{1}+ \omega_{2}+\ldots+\omega_{N}=\xi_1,\\ 
&&\omega_{N+1}+\omega_{N+2}=\xi_2 \\
 &&\xi_1+\xi_2=\mathscr{L}'(1). 
\end{eqnarray*}
Hence we can rewrite the equations \eqref{eq:moment_equation} as
\begin{subequations}\label{eq:moment_equation_new}
\begin{eqnarray}
&&\left\{
  \begin{array}{ll}
\omega_{1}+ \omega_{2}+\ldots+\omega_{N}=\xi_1, \label{eq:moment_equation2}\\
\omega_1v_{1,i}+\omega_2v_{2,i}+\ldots+\omega_{N}v_{N,i}=\mathscr{L}'(t_i),\quad i=1,2,\ldots,n-1 \\
\omega_1v_{1,i}v_{1,j}+\omega_2v_{2,i}v_{2,j}+\ldots+\omega_{N}v_{N,i}v_{N,j}=\mathscr{L}'(t_it_j),\quad i,j=1,2,\ldots,n-1 \\
\omega_1v_{1,i}v_{1,j}v_{1,k}+\ldots+\omega_{N}v_{N,i}v_{N,j}v_{N,k}=\mathscr{L}'(t_it_jt_k),\quad i,j,k=1,2,\ldots,n-1 \\
\end{array}
\right. \label{eq:cuba_dimn-1}\\
&&\left\{
  \begin{array}{ll}
\omega_{N+1}+\omega_{N+2}=\xi_2 \\
\omega_{N+1}v_{N+1,n}+\omega_{N+2}v_{N+2,n}=\mathscr{L}'(t_n)\\
 \omega_{N+1}v_{N+1,n}^2+\omega_{N+2}v_{N+2,n}^2=\mathscr{L}'(t_n^2)\\
 \omega_{N+1}v_{N+1,n}^3+\omega_{N+2}v_{N+2,n}^3=\mathscr{L}'(t_n^3)
\end{array}
\right. \label{eq:cuba_dim1}\\
&&\quad\ \xi_1+\xi_2=\mathscr{L}'(1). \label{eq:xi}
\end{eqnarray}
\end{subequations}
Once $\xi_i$ is determined by \eqref{eq:xi}, then \eqref{eq:cuba_dimn-1} and
\eqref{eq:cuba_dim1} become one $n-1$ dimensional and the other one-dimensional moment
problems respectively. If these two lower dimensional moment problems can be solved, then
we can get a cubature formula of degree 3 with respect to the original integration
problem. Generally speaking, the one-dimensional moment problem can be easily solved, but
it is difficult to be solved for the $n-1$ dimensional moment problem. However, if the
$n-1$ dimensional problem can be divided into one  $n-2$ dimensional moment problem and
the other one-dimensional moment problem and further the $n-2$ dimensional moment problem
can continue this process, then the original integration problem can be turned into $n$
one-dimensional moment problems.

In what follows, we will prove that if $\mathscr{L}$ is a integral functional with
permutation symmetry, then the construction problem of third-degree cubature formulae can
be turned into $n$ one-dimensional moment problems. In fact, we usually encounter this
kind of integral functional, for example, the integration over the simplex, the square,
the ball or the positive sector of the ball, that is $\{(x_1,x_2,\ldots,x_n)|x_i\geq
0,i=1,2,\ldots,n; x_1^2+x_2^2+\ldots,+x_n^2\leq 1\}$.

We first prove that the integral functional with permutation
symmetry must meet the property (P). Thus the original cubature
problem can be turned into two sub-cubature problems.

\begin{thm}\label{th:l1ln}
Let $n\geq 2$ and
\begin{eqnarray*}
&&l_i(x_1,\ldots,x_n)=x_i-x_n,\quad i=1,2,\ldots,n-1\\
&&l_n(x_1,\ldots,x_n)=x_1+x_2+\ldots+x_n+c_n
\end{eqnarray*}
where
$$
c_n=-\frac{\mathscr{L}(x_1^3+(n-3)x_1^2x_2-(n-2)x_1x_2x_3)}{\mathscr{L}(x_1^2-x_1x_2)}.
$$
  If $\mathscr{L}$ is permutation symmetrical,
  then $l_il_n\ (i=1,2,\ldots,n-1)$ is orthogonal to the polynomials
  of degree $\leq 1$.
\end{thm}
\begin{proof}
  Take $l_1l_n$ as an example. We first prove $\mathscr{L}(x_1^2-x_1x_2)\neq 0$ to confirm the existence of $l_n$. In
  fact, by the symmetry and the positivity,
$$
\mathscr{L}(x_1^2-x_1x_2)=\frac{1}{2}\mathscr{L}(x_1^2-x_1x_2-x_1x_2+x_2^2)=\frac{1}{2}\mathscr{L}((x_1-x_2)^2)>0.
$$
Let us exam the orthogonality of $l_1l_n$. Assume $n\geq 3$. By the symmetry, we have
  \begin{eqnarray*}
  &&\mathscr{L}(l_1l_n)=\mathscr{L}\big(x_1^2-x_n^2+(x_1-x_n)(x_2+\ldots+x_{n-1}+c_n)\big)=0,\\
 &&\mathscr{L}(x_il_1l_n)=\mathscr{L}\big(x_1^2x_i-x_n^2x_i+(x_1-x_n)(x_2x_i+\ldots+x_{n-1}x_i+c_nx_i)\big)=0,\
 2\leq i\leq n-1
  \end{eqnarray*}
and for $i=1$(the same to $i=n$)
 \begin{eqnarray*}
 &&\mathscr{L}(x_1l_1l_n)=\mathscr{L}\big(x_1^3-x_n^2x_1+(x_1-x_n)(x_2x_1+\ldots+x_{n-1}x_1+c_nx_1)\big)\\
 &=&\mathscr{L}\big(x_1^3+(n-3)x_1^2x_2-(n-2)x_1x_2x_3+c_n(x_1^2-x_1x_2)\big)=0.
  \end{eqnarray*}
 It is easy to verify that the result holds when $n=2$. This completes the proof.
\end{proof}

Let $\mathscr{L}_1$ and $\mathscr{L}^1$ be two linear functionals
defined on $\Pi^{n-1}$ and $\Pi^{1}$, whose moments are determined
by
\begin{eqnarray}
&&\mathscr{L}_1(t_it_jt_k)=\mathscr{L}(l_il_jl_k),\
\mathscr{L}_1(t_it_j)=\mathscr{L}(l_il_j),\
\mathscr{L}_1(t_i)=\mathscr{L}(l_i),\ \mathscr{L}_1(1)=\xi_1^{(0)},\\
&&\mathscr{L}^1(t_n^3)=\mathscr{L}(l_n^3),\
\mathscr{L}^1(t_n^2)=\mathscr{L}(l_n^2),\
\mathscr{L}^1(t_n)=\mathscr{L}(l_n),\ \mathscr{L}^1(1)=\xi_2^{(0)}\\
&&\hspace{2cm}\xi_1^{(0)}+\xi_2^{(0)}=\mathscr{L}(1) \quad\text{and}
\quad 1\leq i,j,k\leq n-1 \notag
\end{eqnarray}
respectively. Thus the construction problem of third-degree
formulas with respect to $\mathscr{L}$ is turned into two smaller
problems, one of which is the construction with respect to
$\mathscr{L}_1$ and the other of which is the construction with
respect to $\mathscr{L}^1$. It is easy to compute
\begin{eqnarray*}
&& \mathscr{L}^1({t_n})=\mathscr{L}(l_n)=n\cdot
\mathscr{L}(x_1)+c_n\mathscr{L}(1),\\
&&
\mathscr{L}^1({t_n}^2)=n\mathscr{L}(x_1^2)+n(n-1)\mathscr{L}(x_1x_2)+2nc_n\cdot
\mathscr{L}(x_1)+c_n^2\mathscr{L}(1)\\
&&
\mathscr{L}^1({t_n}^3)=n\mathscr{L}(x_1^3)+6\binom{n}{2}\mathscr{L}(x_1^2x_2)+6\binom{n}{3}\mathscr{L}(x_1x_2x_3)\\
&&\hspace{1.7cm}+3c_n\big(n\mathscr{L}(x_1^2)+n(n-1)\mathscr{L}(x_1x_2)\big)+3nc_n^2\mathscr{L}(x_1)+c_n^3\mathscr{L}(1)\\
&&\hspace{1.4cm}=n\mathscr{L}(x_1^3)+3n(n-1)\mathscr{L}(x_1^2x_2)+n(n-1)(n-2)\mathscr{L}(x_1x_2x_3)\\
&&\hspace{1.7cm}+3c_n\big(n\mathscr{L}(x_1^2)+n(n-1)\mathscr{L}(x_1x_2)\big)+3nc_n^2\mathscr{L}(x_1)+c_n^3\mathscr{L}(1).\\
\end{eqnarray*}

Next we will show this decomposition process can continue.

Let us consider the cubature formula with respect to
$\mathscr{L}_1$. Obviously, $\mathscr{L}_1$ is also permutation
symmetrical, which allows us to employ theorem \ref{th:l1ln}
continuously. Define
\begin{eqnarray*}
&&l^{(1)}_i(t_1,\ldots,t_{n-1})=t_i-t_{n-1},\quad i=1,2,\ldots,n-2,\\
&&l^{(1)}_{n-1}(t_1,\ldots,t_n)=t_1+t_2+\ldots+t_{n-1}+c_{n-1},
\end{eqnarray*}
where
$$
c_{n-1}=-\frac{\mathscr{L}_1(t_1^3+(n-4)t_1^2t_2-(n-3)t_1t_2t_3)}{\mathscr{L}_1(t_1^2-t_1t_2)},
$$
then $l^{(1)}_il^{(1)}_{n-1},n=1,2,\ldots,n-2$ are the orthogonal
polynomials of degree two with respect to $\mathscr{L}_1$. Noticing
$t_i=l_i(x_1,x_2,\ldots,x_n)$, we have
\begin{eqnarray*}
&&l^{(1)}_i(t_1,\ldots,t_{n-1})=t_i-t_{n-1}=x_i-x_{n-1},\quad i=1,2,\ldots,n-2,\\
&&l^{(1)}_{n-1}(t_1,\ldots,t_n)=t_1+t_2+\ldots+t_{n-1}+c_{n-1} \\
&& \hspace{2.6cm}=x_1+x_2+\ldots+x_{n-1}-(n-1)x_n+c_{n-1}.
\end{eqnarray*}
Again let $\mathscr{L}_2$ and $\mathscr{L}^2$ be two linear
functionals defined on $\Pi^{n-2}$ and $\Pi^{1}$, whose moments are
determined by
\begin{eqnarray*}
&&\mathscr{L}_2(\tau_i\tau_j\tau_k)=\mathscr{L}_1(l_i^{(1)}l_j^{(1)}l_k^{(1)}),\
\mathscr{L}_2(\tau_i\tau_j)=\mathscr{L}_1(l_i^{(1)}l_j^{(1)}),\
\mathscr{L}_2(\tau_i)=\mathscr{L}_1(l_i^{(1)}),\ \mathscr{L}_2(1)=\xi_1^{(1)},\\
&&\mathscr{L}^2(\tau_{n-1}^3)=\mathscr{L}_1((l_{n-1}^{(1)})^3),\
\mathscr{L}^2(\tau_{n-1}^2)=\mathscr{L}_1((l_{n-1}^{(1)})^2),\
\mathscr{L}^2(\tau_{n-1})=\mathscr{L}_1(l_n^{(1)}),\ \mathscr{L}^2(1)=\xi_2^{(1)}\\
&&\hspace{2cm}\xi_1^{(1)}+\xi_2^{(1)}=\xi_1^{(0)} \quad\text{and}
\quad 1\leq i,j,k\leq n-2
\end{eqnarray*}
respectively. It is easy to compute
\begin{eqnarray*}
&& \mathscr{L}^2({\tau_{n-1}})=\mathscr{L}_1(l_{n-1}^{(1)})=\mathscr{L}(l_{n-1}^{(1)})+c_{n-1}\big(\mathscr{L}_1(1)-\mathscr{L}(1)\big),\\
&&
\mathscr{L}^2({\tau_{n-1}}^2)=\mathscr{L}_1((l_{n-1}^{(1)})^2)=\mathscr{L}_1\big([(t_1+t_2+\ldots+t_{n-1})+c_{n-1}]^2\big)\\
&&\hspace{1.9cm} =\mathscr{L}((l_{n-1}^{(1)})^2)+c_{n-1}^2\big(\mathscr{L}_1(1)-\mathscr{L}(1)\big),             \\
&& \mathscr{L}^2({\tau_{n-1}}^3)= \mathscr{L}((l_{n-1}^{(1)})^3)+c_{n-1}^3\big(\mathscr{L}_1(1)-\mathscr{L}(1)\big).\\
\end{eqnarray*}

Assume that $\mathscr{L}_{k}$ is a linear functional defined on
$\Pi^{n-k}$ for every $k(0\leq k<n)$  and satisfies property (P).
Then a cubature problem of degree 3 with respect to
$\mathscr{L}_{k}$ can be divided into two smaller cubature
problems---one with respect to $\mathscr{L}_{k+1}$ and the other
with respect to $\mathscr{L}^{k+1}$. Moreover $\mathscr{L}_{k+1}$
also satisfies the property (P) and then this process can continue
and will end when $k=n-1$. Finally, an $n$-dimensional cubature
problem can be transformed into $n$ one-dimensional cubature
problems.

\begin{thm}\label{th:cn}
  Let
$$
l_{n-k}^{(k)}=x_1+x_2+\ldots+x_{n-k}-(n-k)x_{n-k+1}+c_{n-k}
$$
and  $\mathscr{L}^{k}$ be a linear functional defined on $\Pi^1$
according to the above process, then the corresponding moments are
\begin{eqnarray*}
  \mathscr{L}^{k+1}(1) &=& \xi^{(k)}_{2}, \\
  \mathscr{L}^{k+1}(\tau) &=& \mathscr{L}(l_{n-k}^{(k)})+c_{n-k}\big(\mathscr{L}_{k}(1)-\mathscr{L}(1)\big) \\
  &=& c_{n-k}\mathscr{L}_{k}(1),\\
  \mathscr{L}^{k+1}(\tau^2) &=& \mathscr{L}((l_{n-k}^{(k)})^2)+c_{n-k}^2\big(\mathscr{L}_{k}(1)-\mathscr{L}(1)\big) \\
  &=& (n-k)(n-k+1)\mathscr{L}(x_1^2-x_1x_2)+c_{n-k}^2\mathscr{L}_k(1),\\
  \mathscr{L}^{k+1}(\tau^3) &=&
  \mathscr{L}((l_{n-k}^{(k)})^3)+c_{n-k}^3\big(\mathscr{L}_{k}(1)-\mathscr{L}(1)\big)\\
  &=&
  (n-k)(n-k+1)(n-k+2)\mathscr{L}(-x_1^3+3x_1^2x_2-2x_1x_2x_3)+c_{n-k}^3\mathscr{L}_{k}(1)
\end{eqnarray*}
where
\begin{eqnarray}\label{eq:cnk}
c_{n-k}=-\dfrac{\mathscr{L}(x_1^3-3x_1^2x_2+2x_1x_2x_3)}{\mathscr{L}(x_1^2-x_1x_2)},\quad
k=1,2,\ldots,n-2.
\end{eqnarray}
\end{thm}
\begin{proof}
It remains to prove Eq.\eqref{eq:cnk}. It follows from theorem
\ref{th:l1ln} that
$$
c_{n-k}=-\frac{\mathscr{L}_k(t_1^3+(n-3-k)t_1^2t_2-(n-2-k)t_1t_2t_3)}{\mathscr{L}_k(t_1^2-t_1t_2)}.
$$
According to the definition of $\mathscr{L}_k$, we have
\begin{eqnarray*}
&&\mathscr{L}_k(t_1^3+(n-3-k)t_1^2t_2-(n-2-k)t_1t_2t_3)\\
&=&\mathscr{L}(t_1^3+(n-3-k)t_1^2t_2-(n-2-k)t_1t_2t_3)\\
&=&\mathscr{L}\Big((x_1-x_{n-k})^3+(n-3-k)(x_1-x_{n-k})^2(x_2-x_{n-k})\\
&&\hspace{2cm}-(n-2-k)(x_1-x_{n-k})(x_2-x_{n-k})(x_3-x_{n-k})\Big)\\
&=&\mathscr{L}\Big(x_1^3-3x_1^2x_2+2x_1x_2x_3\Big),
\end{eqnarray*}
\begin{eqnarray*}
\mathscr{L}_k(t_1^2-t_1t_2)&=&\mathscr{L}((x_1-x_{n-k})^2-(x_1-x_{n-k})(x_2-x_{n-k}))\\
&=& \mathscr{L}(x_1^2-x_1x_2),
\end{eqnarray*}
where the permutation symmetry is used. This completes the proof.
\end{proof}
\begin{rem}
    In fact $\mathscr{L}_{n-1}$ is also a one-dimensional integration
    functional. The corresponding moment can be calculated by
    \begin{eqnarray}
        \begin{split}
            &\mathscr{L}_{n-1}(1)=\xi_1^{(n-2)}, \\
            &\mathscr{L}_{n-1}(\tau)=\mathscr{L}(x_1-x_2)=0, \\
            &\mathscr{L}_{n-1}(\tau^2)=\mathscr{L}((x_1-x_2)^2)=2\mathscr{L}(x_1^2-x_1x_2), \\
            &\mathscr{L}_{n-1}(\tau^3)=\mathscr{L}((x_1-x_2)^3)=0. \\
        \end{split}
        \label{eq:moment_nth}
    \end{eqnarray}
    For convenience, in what follows let
    $\mathscr{L}^n=\mathscr{L}_{n-1}$.
\end{rem}

Suppose that
\begin{eqnarray}\label{eq:OneDim}
\mathscr{L}^{k}(g)\approx \sum_{i=1}^{n_k}w_{i,k}g(t_{i,k}),
k=1,2,\ldots,n
\end{eqnarray}
is exact for any $g\in \Pi_3^1$. And let
$v^{i,k}=(v_{i,k}^{(1)},v_{i,k}^{(2)},\ldots,v_{i,k}^{(n)})$ be the
solution of
\begin{align}
&\left\{
  \begin{array}{l}
    x_1+x_2+\ldots+x_n+c_n=t_{i,1} \\
  x_{n-1}-x_n=0 \\
  \cdots\cdots\cdots\cdots\cdots\cdots\cdots\cdots \\
  x_{2}-x_n=0\\
   x_1-x_{n}=0
  \end{array}
\right. \qquad\hbox{for $k=1$,}\label{eq:Solution1}\\
 &\left\{
  \begin{array}{l}
    x_1+x_2+\ldots+x_n+c_n=0 \\
  x_1+x_2+\ldots+x_{n-1}-(n-1)x_n+c_{n-1}=0 \\
  \cdots\cdots\cdots\cdots\cdots\cdots\cdots\cdots \\
  x_1+x_2-2x_3+c_{2}=0\\
   x_1-x_2=t_{i,n}
  \end{array}
 \right.\qquad\hbox{for $k=n$,}\label{eq:Solution3}\\
\intertext{and}
 &\left\{
  \begin{array}{l}
    x_1+x_2+\ldots+x_n+c_n=0 \\
  x_1+x_2+\ldots+x_{n-1}-(n-1)x_n+c_{n-1}=0 \\
  \cdots\cdots\cdots\cdots\cdots\cdots\cdots\cdots \\
  x_1+x_2+\ldots+x_{n-k+1}-(n-k+1)x_{n-k+2}+c_{n-k+1}=t_{i,k}\\
  x_{n-k}-x_{n-k+1}=0\\
  \cdots\cdots\cdots\cdots\cdots\cdots\cdots\cdots \\
   x_1-x_{n-k+1}=0
  \end{array}
\right.\qquad\hbox{for $k=2,3,\ldots,n-1$,}\label{eq:Solution2}
\end{align}

Hence the final cubature formula can be written as
\begin{eqnarray*}
\mathscr{L}(f)\approx\sum_{k=1}^n\sum_{i=1}^{n_k}w_{i,k}f(v^{i,k})
\end{eqnarray*}
which is exact for any $f\in \Pi^n_3$. It is clear that the solution
of Eq.\eqref{eq:Solution1} is
$$
(\eta_i,\eta_i,\ldots,\eta_i), \quad \eta_i=\frac{t_{i,1}-c_n}{n}
$$
and the solution of Eq.\eqref{eq:Solution2} is
\begin{eqnarray*}
\left\{
  \begin{array}{ll}
   x_n=\dfrac{c_{n-1}-c_n-\delta_{2,k}t_{i,2}}{n},  \\
   x_{n-1}=x_n+\dfrac{c_{n-2}-c_{n-1}}{n-1}=x_n \\
   \ldots\ldots\ldots \\
   x_{n-k+3}=x_{n-k+4}\\
   x_{n-k+2}=x_{n-k+3}+\dfrac{c_{n-k+2}-c_{n-k+3}-t_{i,k}}{n-k+2}=x_{n-k+3}-\dfrac{t_{i,k}}{n-k+2}\\
   x_1=x_2=\ldots=x_{n-k+1}=x_{n-k+2}+\dfrac{t_{i,k}-c_{n-k+1}}{n-k+1}
  \end{array}
\right.
\end{eqnarray*}
where $\delta_{2,k}=1$ if $k=2$ and $\delta_{2,k}=0$ if $k\neq 2$,
and the solution of Eq.\eqref{eq:Solution3} is
\begin{eqnarray*}
\left\{
  \begin{array}{ll}
   x_n=x_{n-1}=\ldots=x_3=\dfrac{c_{n-1}-c_n}{n} \\
   x_{2}=x_3-\dfrac{c_{n-1}+t_{i,n}}{2}\\
   x_1=x_2+t_{i,n}=x_3+\dfrac{t_{i,n}-c_{n-1}}{2}
  \end{array}
\right.
\end{eqnarray*}

Collecting the above discussion, we have
\begin{thm}\label{th:coef}
Assume that $\mathscr{L}$ is permutation symmetrical. Then there
must exist a cubature formula
\begin{eqnarray*}
\mathscr{L}(f)&\approx
&\sum_{k=2}^{n}\sum_{i=1}^{m_k}w_{i,k}f(\underbrace{\alpha_{i,k},\ldots,\alpha_{i,k}}_{n-k+1},\beta_{i,k},\underbrace{\gamma,\gamma,\ldots,\gamma}_{k-2})+\sum_{i=1}^{m_1}w_{i,1}f(\alpha_{i,1},\ldots,\alpha_{i,1}):=C(f)\\
\end{eqnarray*}
which is exact for every polynomial of degree $\leq 3$. In the
formula, $\alpha$'s and $\beta$'s can be computed by
\begin{eqnarray*}
\left\{
  \begin{array}{ll}
   \gamma=\dfrac{c_{n-1}-c_n}{n}\\
    \beta_{i,k}=\gamma-\dfrac{t_{i,k}}{n-k+2}, & \hbox{$2\leq k \leq n-1$;} \\
     \alpha_{i,k}=\beta_{i,k}+\dfrac{t_{i,k}-c_{n-k}}{n-k+1}, & \hbox{$2\leq k \leq n-1$;} \\
    \alpha_{i,1}=\dfrac{t_{i,1}-c_{n}}{n}, & \\
    \beta_{i,n}=\gamma-\dfrac{t_{i,n}+c_2}{2}\\
    \alpha_{i,n}=\beta_{i,n}+t_{i,n}
  \end{array}
\right.
\end{eqnarray*}
where $t_{i,k}$'s and $w_{i,k}$'s are the nodes and weights of the
quadrature formula \eqref{eq:OneDim} with respect to
$\mathscr{L}^k$.
\end{thm}

The proof is a direct result of the computation and is omitted.

\begin{rem}
  For the quadrature problem of the one-dimensional moment, it is
  well known that the number of the nodes $n_k=2$ in the general
  case. Hence the total number of the nodes of the cubature formula
  with respect to $\mathscr{L}$ is $2n$ generally and $2n$ is
  usually the minimum among all the cubature formula of degree 3 except one case of the integration over the n-dimensional simplex \cite{stroud1961}.
For more knowledge of the problem of the one-dimensional moment,  we
can refer to the appendix of \cite{tnt}.

\end{rem}
\begin{rem}
  For convenience, we present the relations of $\xi$'s as follows
  \begin{eqnarray*}
  \left.
    \begin{array}{ccccccccccc}
      \mathscr{L}(1) & \longrightarrow & \xi_1^{(0)} & \longrightarrow & \xi_1^{(1)} &  \longrightarrow & \ldots & \longrightarrow & \xi_1^{(n-2)}\\
      &  &  + &  &  +  & &&& + \\
      &          &  \xi_2^{(0)} &  & \xi_2^{(1)} & && & \xi_2^{(n-2)}\\
    \end{array}
  \right.
  \end{eqnarray*}
According to the previous discussion, it is clear that
$$
\mathscr{L}^k(1)=\xi_2^{(k-1)}\quad \text{for}\quad 1\leq k\leq n-1
\quad \text{and} \quad
\mathscr{L}^n(1)=\mathscr{L}_{n-1}(1)=\xi_1^{(n-2)}.
$$
\end{rem}

\section{Numerical Examples}
$\bullet$ Firstly take the integration over the n-dimensional simplex as
an example. Define
$$
\mathscr{L}(f)=\int_{T_n}f(x_1,x_2,\ldots,x_n)\mathrm{d}x_1\mathrm{d}x_2\ldots\mathrm{d}x_n
$$
where
$$
T_n=\{(x_1,x_2,\ldots,x_n)|x_1+x_2+\ldots+x_n\leq 1,\quad x_i\geq
0,\quad i=1,2,\ldots,n\}.
$$
It is well known that
$$
\mathscr{L}(x_1^{\alpha_1}x_2^{\alpha_2}\ldots
x_n^{\alpha_n})=\dfrac{\alpha_1!\alpha_2!\ldots
\alpha_n!}{(n+\alpha_1+\alpha_2+\ldots+\alpha_n)!}.
$$
Then by a simple computation, we have
\begin{eqnarray*}
&&c_i=-\frac{2}{n+3},i=2,3,\ldots,n-1, \quad c_n=-\frac{n+2}{n+3},\\
&&\gamma=\frac{1}{n+3}.
\end{eqnarray*}
Here if we take
$$
\xi_2^{(i)}=t_{i+1}\cdot \frac{1}{n\cdot n!},\ \xi_1^{(n-2)}=t_n \cdot \frac{1}{n\cdot
n!}, \ i=0,1,\ldots,n-2,
$$
and $\sum_{i=1}^n t_i=n$, then the moments of
$\mathscr{L}^{k+1}(1\leq k\leq n-2)$ are
\begin{eqnarray*}
  \mathscr{L}^{k+1}(1) &=& t_{k+1}\frac{1}{n\cdot
n!}, \\
  \mathscr{L}^{k+1}(\tau) &=& c_{n-k}\mathscr{L}_{k}(1)=-\frac{2(n-\sum_{i=1}^k t_k)}{(n+3)n\cdot
n!}, \\
  \mathscr{L}^{k+1}(\tau^2) &=& \mathscr{L}((l_{n-k}^{(k)})^2)+c_{n-k}^2\big(\mathscr{L}_{k}(1)-\mathscr{L}(1)\big)=\frac{(n-k)(n-k+1)}{(n+2)!}+\frac{4(n-\sum_{i=1}^k t_k)}{n(n+3)^2\cdot
n!}, \\
  \mathscr{L}^{k+1}(\tau^3) &=&
  (n-k)(n-k+1)(n-k+2)\mathscr{L}(-x_1^3+3x_1^2x_2-2x_1x_2x_3)+c_{n-k}^3\mathscr{L}_{k}(1)\\
  &=&\frac{-2(n-k)(n-k+1)(n-k+2)}{(n+3)!}-\frac{8(n-\sum_{i=1}^k t_k)}{n(n+3)^3\cdot
n!}
\end{eqnarray*}
and
\begin{eqnarray*}
  \mathscr{L}^{1}(1) &=& \frac{t_1}{n\cdot
n!}, \\
  \mathscr{L}^{1}(\tau) &=& \frac{-2}{(n+1)!(n+3)}, \\
  \mathscr{L}^{1}(\tau^2) &=& \frac{n^2+5n+8}{(n+3)!(n+3)}, \\
  \mathscr{L}^{1}(\tau^3) &=&
  -\frac{2(n+2)(n+4)}{(n+3)^2(n+3)!}
\end{eqnarray*}
and
\begin{eqnarray*}
  \mathscr{L}^{n}(1) &=& \frac{t_n}{n\cdot
n!}, \\
  \mathscr{L}^{n}(\tau) &=& 0, \\
  \mathscr{L}^{n}(\tau^2) &=& \frac{2}{(n+2)!}, \\
  \mathscr{L}^{n}(\tau^3) &=&
  0.
\end{eqnarray*}
By taking different values for $\xi$s, we can get different cubature
formulae. For $n=3$ and $n=4$, if we take all $t_i=1$, then we can
get formulas as showed in Tables \ref{tab:t31} and \ref{tab:t41}.
\begin{table}[h]
  \centering
  \begin{tabular}{|cccc|}
  \hline
  $x_1$ & $x_2$ & $x_3$ & weight \\ \hline
   0.34240723692377 &   0.34240723692377 &   0.34240723692377  &  0.01469064053612  \\
   0.14125289379518 &   0.14125289379518 &   0.14125289379518  &  0.04086491501944  \\
   0.41353088165296 &   0.41353088165296 &   0.00627157002742  &  0.01887111233337  \\
   0.12380973765487 &   0.12380973765487 &   0.58571385802358  &  0.03668444322218  \\
   0.60719461208592 &   0.05947205458075 &   0.16666666666667  &  0.02777777777778  \\
   0.05947205458075 &   0.60719461208592 &   0.16666666666667  &  0.02777777777778  \\
  \hline
\end{tabular}
  \caption{Nodes and weights for $T_3$}\label{tab:t31}
\end{table}

\begin{table}[h]
  \centering
\begin{tabular}{|cccc|}
  \hline
  $x_1$ & $x_2$ & $x_3$ & $x_4$  \\ \hline
   0.27145760185760 &  0.27145760185760 &  0.27145760185760 &  0.27145760185760     \\
   0.12024746726682 &  0.12024746726682 &  0.12024746726682 &  0.12024746726682     \\
   0.30652570925957 &  0.30652570925957 &  0.30652570925957 & -0.06243427063585     \\
   0.11154151763119 &  0.11154151763119 &  0.11154151763119 &  0.52251830424930     \\
   0.37131176827505 &  0.37131176827505 & -0.02833782226438 &  0.14285714285714     \\
   0.09266869570542 &  0.09266869570542 &  0.52894832287488 &  0.14285714285714     \\
   0.54391317546145 &  0.02751539596712 &  0.14285714285714 &  0.14285714285714     \\
   0.02751539596712 &  0.54391317546145 &  0.14285714285714 &  0.14285714285714     \\
  \hline
  \multicolumn{4}{|c|}{$w_1=0.00254167472911,\ w_2=
0.00787499193755,\ w_3=0.00294495824332,$}\\
\multicolumn{4}{|c|}{$w_4=0.00747170842335,\ w_5=0.00365639117145,\
w_6=0.00676027549522,$}\\ \multicolumn{4}{|c|}{$\
w_7=0.00520833333333,\ w_8=0.00520833333333.$}    \\ \hline
\end{tabular}
  \caption{Nodes and weights for $T_4$}\label{tab:t41}
\end{table}

In Table \ref{tab:t31}, the first point is outside of the region. To
avoid it, we can take
$$
t_1=\frac{93}{85},\ t_2=\frac{378}{391},\ t_3=\frac{108}{115},
$$
and the corresponding cubature formula is listed in Table
\ref{tab:t32}. In the formula, the first and third nodes are on the
boundary of the region $T_3$.

\begin{table}[h]
\centering
\begin{tabular}{|cccc|}
  \hline
  $x_1$ & $x_2$ & $x_3$ & weight \\ \hline
   0.33333333333333 &   0.33333333333333 &   0.33333333333333  &  0.01875000000000  \\
   0.14285714285714 &   0.14285714285714 &   0.14285714285714  &  0.04203431372549  \\
   0.41666666666667 &   0.41666666666667 &   0.00000000000000  &  0.01875000000000  \\
   0.12037037037037 &   0.12037037037037 &   0.59259259259259  &  0.03495843989770  \\
   0.61593041596355 &   0.05073625070311 &   0.16666666666667  &  0.02608695652174  \\
   0.05073625070311 &   0.61593041596355 &   0.16666666666667  &  0.02608695652174  \\
  \hline
\end{tabular}
\caption{Nodes and weights for $T_3$}\label{tab:t32}
\end{table}

If we take
$$
t_1=\frac{94}{85},\ t_2=1,\ t_3=\frac{76}{85},
$$
then all the nodes are inside the region, see Table \ref{tab:t33}.

\begin{table}[h]
\centering
\begin{tabular}{|cccc|}
  \hline
  $x_1$ & $x_2$ & $x_3$ & weight \\ \hline
   0.33237874197689 &   0.33237874197689 &   0.33237874197689  &  0.01927056497746  \\
   0.14303319621635 &   0.14303319621635 &   0.14303319621635  &  0.04216734351927  \\
   0.41247250927755 &   0.41247250927755 &   0.00838831477823  &  0.02043165185637  \\
   0.12085734331926 &   0.12085734331926 &   0.59161864669481  &  0.03512390369919  \\
   0.61593041596355 &   0.04371016618787 &   0.16666666666667  &  0.02516339869281  \\
   0.04371016618787 &   0.61593041596355 &   0.16666666666667  &  0.02516339869281  \\
  \hline
\end{tabular}
\caption{Nodes and weights for $T_3$}\label{tab:t33}
\end{table}

In Table \ref{tab:t41}, all the weights are positive. However, we
find there exist three points outside of the region. If we take
$$
t_1=\frac{104}{75},\ t_2=\frac{3577}{2775},\ t_3=\frac{9947}{8880},\
t_4=\frac{49}{60}.
$$
and add one more point with weight $-49/80$, then the corresponding
formula is showed in Table \ref{tab:t42}.

\begin{table}[h]
  \centering
\begin{tabular}{|cccc|}
  \hline
  $x_1$ & $x_2$ & $x_3$ & $x_4$  \\ \hline
   0.25000000000000 &  0.25000000000000 &  0.25000000000000 &  0.25000000000000     \\
   0.12500000000000 &  0.12500000000000 &  0.12500000000000 &  0.12500000000000     \\
   0.28571428571429 &  0.28571428571429 &  0.28571428571429 &  0.00000000000000     \\
   0.11224489795918 &  0.11224489795918 &  0.11224489795918 &  0.52040816326531     \\
   0.35714285714286 &  0.35714285714286 &  0.00000000000000 &  0.14285714285714     \\
   0.07792207792208 &  0.07792207792208 &  0.55844155844156 &  0.14285714285714     \\
   0.00000000000000 &  0.42857142857143 &  0.14285714285714 &  0.14285714285714     \\
   0.42857142857143 &  0.00000000000000 &  0.14285714285714 &  0.14285714285714     \\
   0.28571428571429 &  0.28571428571429 &  0.14285714285714 &  0.14285714285714     \\
  \hline
 \multicolumn{4}{|c|}{$w_1=0.00555555555556,\ w_2=
                           0.00888888888889,\
                       w_3=0.00600490196078,$}\\
\multicolumn{4}{|c|}{$w_4=0.00742227521640,\
                      w_5=0.00633074935401,\
                      w_6=0.00533755957993,$}\\
\multicolumn{4}{|c|}{$\
                      w_7=0.00425347222222,\ w_8=0.00425347222222, w_9=-49/80.$}    \\
\hline
\end{tabular}
  \caption{Nodes and weights for $T_4$}\label{tab:t42}
\end{table}

In Table \ref{tab:t42}, there are 5 points on the boundary. To make
all the nodes inside the region, we can take
$$
t_1=\frac{7}{5},\ t_2=\frac{187}{145},\ t_3=\frac{179522}{160283},\ t_4=\frac{5}{6}
$$
and add one more node with weight $-618391/961698$ and the corresponding formula is shown
in Table \ref{tab:t43}.

\begin{table}[h]
  \centering
\begin{tabular}{|cccc|}
  \hline
  $x_1$ & $x_2$ & $x_3$ & $x_4$  \\ \hline
   0.24955035825246 &  0.24955035825246 &  0.24955035825246 &  0.24955035825246     \\
   0.12511271452382 &  0.12511271452382 &  0.12511271452382 &  0.12511271452382     \\
   0.28567845223177 &  0.28567845223177 &  0.28567845223177 &  0.00010750044754     \\
   0.11210697374487 &  0.11210697374487 &  0.11210697374487 &  0.52082193590823     \\
   0.35714280321315 &  0.35714280321315 &  0.00000010785942 &  0.14285714285714     \\
   0.07749399446444 &  0.07749399446444 &  0.55929772535683 &  0.14285714285714     \\
   0.56855699818890 &  0.00287157323967 &  0.14285714285714 &  0.14285714285714     \\
   0.00287157323967 &  0.56855699818890 &  0.14285714285714 &  0.14285714285714     \\
   0.28571428571429 &  0.28571428571429 &  0.14285714285714 &  0.14285714285714     \\
  \hline
 \multicolumn{4}{|c|}{$w_1=0.00566710734383,\ w_2=0.0089162259895080,\
w_3=0.006035977209200,$}\\
\multicolumn{4}{|c|}{$w_4=0.00739793083678,\
w_5=0.0063627387083462,\ w_6=0.005780572945942,$}\\
\multicolumn{4}{|c|}{$\
w_7=0.00434027777778,\ w_8=0.0043402777777778,\ w_9=-0.643019950129875.$}    \\
\hline
\end{tabular}
  \caption{Nodes and weights for $T_4$}\label{tab:t43}
\end{table}

The integration problem on the $n$-simplex was studied very
extensively. According to the collection of R. Cools in the
website(http://nines.cs.kuleuven.be/research/ecf/ecf.html), the
minimum number of nodes in the third-degree formulas is $n+2$, in
which there is a negative weight. Except the $(n+2)$-point formula,
the minimum number is 8 and 10 for $n=3$ and $n=4$, respectively. If
we only consider the formulas with positive weight, the minimum of
the points is 8 and 11 for $n=3$ and $n=4$ respectively. Therefore
our formula for $n=3$ and $n=4$ have the fewest numbers among all
the formulas with positive weights.

$\bullet$ Secondly take the integration over the positive sector of a ball as an example.
Define
$$
\mathscr{L}(f)=\int_{S_n}f(x_1,x_2,\ldots,x_n)\mathrm{d}x_1\mathrm{d}x_2\ldots\mathrm{d}x_n
$$
where
$S_n=\{(x_1,x_2,\ldots,x_n)|x_1^2+x_2^2+\ldots+x_n^2\leq 1, x_1\geq 0,\ldots,
x_n\geq 0\}$. It is easy to get that
$$
\mathscr{L}(x_1^{\alpha_1}x_2^{\alpha_2}\ldots
x_n^{\alpha_n})=\dfrac{(\alpha_1-1)!!(\alpha_2-1)!!\ldots
(\alpha_n-1)!!}{(n+\alpha_1+\alpha_2+\ldots+\alpha_n)!!}\cdot
\Big(\dfrac{\pi}{2}\Big)^{\big[\frac{n-n_o}{2}\big]}
$$
where $n_o$ denotes the number of the odd number among
$\alpha_1,\alpha_2,\ldots,\alpha_n$ and $m!!$ denotes the double
factorial of $m$ and $m!!=1$ if $m\leq 0$.
Then by a simple computation, we have
\begin{eqnarray*}
    &&c_i=-\dfrac{(n+2)!!}{(n+3)!!}\cdot\Big(\dfrac{\pi}{2}\Big)^{\big[\frac{n-1}{2}\big]-\big[\frac{n}{2}\big]}\cdot
    \dfrac{4-\pi}{\pi-2},i=2,3,\ldots,n-1, \\
    && c_n=-\dfrac{(n+2)!!}{(n+3)!!}\cdot\Big(\dfrac{\pi}{2}\Big)^{\big[\frac{n-1}{2}\big]-\big[\frac{n}{2}\big]}\cdot
    \dfrac{(n-1)\pi-(2n-4)}{\pi-2},\\
&&\gamma=\dfrac{(n+2)!!}{(n+3)!!}\cdot\Big(\dfrac{\pi}{2}\Big)^{\big[\frac{n-1}{2}\big]-\big[\frac{n}{2}\big]}.
\end{eqnarray*}

Here if we take
$$
\xi_2^{(i)}=\frac{t_{i+1}}{n\cdot n!!}\cdot
\Big(\dfrac{\pi}{2}\Big)^{\big[\frac{n}{2}\big]},\ \xi_1^{(n-2)}=\frac{t_{n}}{n\cdot
n!!}\cdot \Big(\dfrac{\pi}{2}\Big)^{\big[\frac{n}{2}\big]},\ i=0,1,2,\ldots,n-2,
$$
then the moments of $\mathscr{L}^{k+1}(1\leq k\leq n-2)$ are
\begin{eqnarray*}
  \mathscr{L}^{k+1}(1) &=& \frac{t_{k+1}}{n\cdot
  n!!}\cdot \big(\frac{\pi}{2}\big)^{[\frac{n}{2}]}, \\
  \mathscr{L}^{k+1}(\tau) &=&
  c_{n-k}\mathscr{L}_{k}(1)=-\dfrac{(n+2)\cdot (n-\sum_{i=1}^kt_i)}{n\cdot (n+3)!!}\cdot\Big(\dfrac{\pi}{2}\Big)^{\big[\frac{n-1}{2}\big]}\cdot    \dfrac{4-\pi}{\pi-2}, \\
  \mathscr{L}^{k+1}(\tau^2) &=&
  \frac{(n-k)(n-k+1)}{(n+2)!!}\Big(\frac{\pi}{2}\Big)^{[\frac{n-2}{2}]}\cdot
  \Big(\frac{\pi}{2}-1\Big)\\
  &&+\frac{(n-\sum_{i=1}^kt_i)[(n+2)!!]^2}{n\cdot n!!\cdot
  [(n+3)!!]^2}\cdot \Big(\frac{4-\pi}{\pi-2}\Big)^2\cdot
  \Big(\frac{\pi}{2}\Big)^{2[\frac{n-1}{2}]-[\frac{n}{2}]}, \\
  \mathscr{L}^{k+1}(\tau^3) &=&
  \frac{(n-k)(n-k+1)(n-k+2)}{(n+3)!!}\Big(\frac{\pi}{2}\Big)^{[\frac{n-3}{2}]}\cdot
  \Big(\frac{\pi}{2}-2\Big)\\
  &&-\frac{(n-\sum_{i=1}^kt_i)[(n+2)!!]^3}{n\cdot n!!\cdot
  [(n+3)!!]^3}\cdot \Big(\frac{4-\pi}{\pi-2}\Big)^3\cdot
  \Big(\frac{\pi}{2}\Big)^{3[\frac{n-1}{2}]-2[\frac{n}{2}]}
\end{eqnarray*}
and
\begin{eqnarray*}
  \mathscr{L}^{1}(1) &=& \frac{t_1}{n\cdot
n!!}\cdot \Big(\dfrac{\pi}{2}\Big)^{\big[\frac{n}{2}\big]}, \\
\mathscr{L}^{1}(\tau) &=& \frac{2(n\pi-3n+\pi-4)}{(n+3)!!\cdot (\pi-2)}\cdot
\Big(\frac{\pi}{2}\Big)^{[\frac{n-1}{2}]}, \\
\mathscr{L}^{1}(\tau^2) &=&
\frac{n(\frac{\pi}{2}+n-1)}{(n+2)!!}\Big(\frac{\pi}{2}\Big)^{[\frac{n-2}{2}]}-\frac{2n\cdot
(n+2)!!}{(n+1)!!\cdot (n+3)!!}\cdot
\Big(\frac{\pi}{2}\Big)^{2[\frac{n-1}{2}]-[\frac{n}{2}]}\cdot
\dfrac{(n-1)\pi-(2n-4)}{\pi-2}\\
&& + \frac{n+2}{[(n+3)!!]^2}\cdot \Big(\frac{\pi}{2}\Big)^{2[\frac{n-1}{2}]-[\frac{n}{2}]}\cdot
\Big(\dfrac{(n-1)\pi-(2n-4)}{\pi-2}\Big)^2,\\
\mathscr{L}^{1}(\tau^3) &=& \frac{3n^2-n}{(n+3)!!}\cdot
\Big(\frac{\pi}{2}\Big)^{[\frac{n-1}{2}]}+\frac{n(n-1)(n-2)}{(n+3)!!}\Big(\frac{\pi}{2}\Big)^{[\frac{n-3}{2}]}\\
&& -\frac{3n(\frac{\pi}{2}+n-1)}{(n+3)!!}\cdot
\Big(\frac{\pi}{2}\Big)^{[\frac{n-3}{2}]}\cdot
\dfrac{(n-1)\pi-(2n-4)}{\pi-2}\\
&& +\frac{3n}{(n+1)!!}\cdot\Big[\frac{(n+2)!!}{(n+3)!!}\Big]^2\cdot
\Big(\frac{\pi}{2}\Big)^{3[\frac{n-3}{2}]-2[\frac{n}{2}]}\cdot\Big(\dfrac{(n-1)\pi-(2n-4)}{\pi-2}\Big)^2\\
&& -\frac{1}{n!!}\cdot\Big[\frac{(n+2)!!}{(n+3)!!}\Big]^3\cdot
\Big(\frac{\pi}{2}\Big)^{3[\frac{n-3}{2}]-2[\frac{n}{2}]}\cdot\Big(\dfrac{(n-1)\pi-(2n-4)}{\pi-2}\Big)^3
\end{eqnarray*}
and
\begin{eqnarray*}
  \mathscr{L}^{n}(1) &=& \frac{t_n}{n\cdot
n!!}\cdot \Big(\dfrac{\pi}{2}\Big)^{\big[\frac{n}{2}\big]}, \\
  \mathscr{L}^{n}(\tau) &=& 0, \\
  \mathscr{L}^{n}(\tau^2) &=&
  \frac{2}{(n+2)!!}\Big(\frac{\pi}{2}\Big)^{[\frac{n-2}{2}]}\cdot
  \Big(\frac{\pi}{2}-1\Big), \\
  \mathscr{L}^{n}(\tau^3) &=&
  0.
\end{eqnarray*}

If take all $t_i=1$ for $n=3$ and $n=4$, then we can get formulae as showed in Tables
\ref{tab:s31} and \ref{tab:s41}. If we take $t_1=0.8, t_2=1.31, t_3=1.11$ and $t_4=0.78$
for $n=4$, then we can get a formula with all the nodes inside the region, see Table
\ref{tab:s42}.

\begin{table}[h]
\centering
\begin{tabular}{|cccc|}
  \hline
  $x_1$ & $x_2$ & $x_3$ & weight \\ \hline
0.53887049476004 & 0.53887049476004 & 0.53887049476004 & 0.07852747507104  \\
0.18341741723402 & 0.18341741723402 & 0.18341741723402 & 0.09600545012840  \\
0.57520979290336 & 0.57520979290336 & 0.02206116228206 & 0.06975676243570  \\
0.20283315000517 & 0.20283315000517 & 0.76681444807844 & 0.10477616276373  \\
0.76016315955181 & 0.09981758853698 & 0.31250000000000 & 0.08726646259972  \\
0.09981758853698 & 0.76016315955181 & 0.31250000000000 & 0.08726646259972  \\
    \hline
\end{tabular}
  \caption{Nodes and weights for $S_3$}\label{tab:s31}
\end{table}

\begin{table}[h]
\centering
\begin{tabular}{|cccc|}
  \hline
  $x_1$ & $x_2$ & $x_3$ & $x_4$  \\ \hline
0.47721483105875  & 0.47721483105875 &  0.47721483105875 &  0.47721483105875  \\
0.17126237887529  & 0.17126237887529 &  0.17126237887529 &  0.17126237887529  \\
0.48420041705925  & 0.48420041705925 &  0.48420041705925 & -0.06966276495181  \\
0.20526869095372  & 0.20526869095372 &  0.20526869095372 &  0.76713241336478  \\
0.56004995494835  & 0.56004995494835 & -0.02818760532449 &  0.29102618165375  \\
0.16531879543241  & 0.16531879543241 &  0.76127471370739 &  0.29102618165375  \\
0.74847573599445  & 0.05241038692400 &  0.29102618165375 &  0.29102618165375  \\
0.05241038692400  & 0.74847573599445 &  0.29102618165375 &  0.29102618165375  \\
\hline
 \multicolumn{4}{|c|}{$w_1=0.03771636146294,\ w_2=0.03938992292057,\ w_3=0.02874740384082,$}\\
\multicolumn{4}{|c|}{$w_4=0.04835888054269,\ w_5=0.03167997303102,\
w_6=0.04542631135249,$}\\ \multicolumn{4}{|c|}{$\
w_7=0.03855314219176,\
0.03855314219176.$}    \\
\hline
\end{tabular}
  \caption{Nodes and weights for $S_4$}\label{tab:s41}
\end{table}

\begin{table}[h]
\centering
\begin{tabular}{|cccc|}
  \hline
  $x_1$ & $x_2$ & $x_3$ & $x_4$  \\ \hline
0.49819378497585  & 0.49819378497585 &  0.49819378497585 &  0.49819378497585  \\
0.15640817934597  & 0.15640817934597 &  0.15640817934597 &  0.15640817934597  \\
0.46048445804733  & 0.46048445804733 &  0.46048445804733 &  0.00148511203764  \\
0.21848509706654  & 0.21848509706654 &  0.21848509706654 &  0.72748319509617  \\
0.54494615070832  & 0.54494615070832 &  0.00202000298217 &  0.29102618165375  \\
0.16998718727254  & 0.16998718727254 &  0.75193793030368 &  0.29102618165375  \\
0.79451246595922  & 0.00637365695922 &  0.29102618165375 &  0.29102618165375  \\
0.00637365695922  & 0.79451246595922 &  0.29102618165375 &  0.29102618165375  \\
\hline
 \multicolumn{4}{|c|}{$w_1=0.02857087469598,\ w_2=0.03311415281083,\ w_3=0.04213154579078,$}\\
\multicolumn{4}{|c|}{$w_4=0.05887768675432,\ w_5=0.03842850515254,\
w_6=0.04715947052132,$}\\ \multicolumn{4}{|c|}{$\
w_7=0.03855314219176,\
w_8=0.03855314219176.$}    \\
\hline
\end{tabular}
  \caption{Nodes and weights for $S_4$}\label{tab:s42}
\end{table}

\section{Conclusion}
In this paper, we present a method to construct third-degree
formulae for integrals with permutation symmetry. Our method is a
decomposition method, which is easy to compute. At the end of the
paper, we present some numerical results, which seem to be new.
Compared with the existing method, we focus on the case of
permutation symmetry, which seldom is considered. The numerical
results show that the number of the points attain of close to the
minimum. Besides, in most cases, the weights in our formulas are all
positive or at most one negative weight.

\end{document}